\documentclass{amsart}

\usepackage{graphicx,bbm}

\newtheorem{theorem}{Theorem}[section]
\newtheorem{lemma}[theorem]{Lemma}

\theoremstyle{definition}
\newtheorem{definition}[theorem]{Definition}
\newtheorem{proposition}[theorem]{Proposition}

\theoremstyle{corollary}

\numberwithin{equation}{section}

\newcommand{\stirlingi}{\genfrac{[}{]}{0pt}{}}
\newcommand{\stirlingii}{\genfrac{\{}{\}}{0pt}{}}

\newcommand{\sgn}{\operatorname{sgn}}

\begin{document}

\title{Central limit theorem for descents in conjugacy classes of $S_n$}

\author{Gene B. Kim}
\address{Department of Mathematics, University of Southern California, Los Angeles, CA 90089}
\email{genebkim@usc.edu}

\author{Sangchul Lee}
\address{Department of Mathematics, University of California, Los Angeles, Los Angeles, CA 90095}
\email{sos440@math.ucla.edu}

\subjclass[2010]{Primary 05A05, 05E99; Secondary 60F05}

\date{March 27, 2018}

\keywords{descents, central limit theorem}

\begin{abstract}
The distribution of descents in fixed conjugacy classes of $S_n$ has been studied, and it is shown that its moments have interesting properties. Fulman proved that the descent numbers of permutations in conjugacy classes with large cycles are asymptotically normal, and Kim proved that the descent numbers of fixed point free involutions are also asymptotically normal. In this paper, we generalize these results to prove a central limit theorem for descent numbers of permutations in any conjugacy class of $S_n$.
\end{abstract}

\maketitle

\section{Introduction}

The Eulerian function $A_n(x)$ was first defined by the relation
\[
	\sum_{j \geq 1} j^n x^{j-1} = \frac{A_n(x)}{(1-x)^{n+1}}
\]
by Euler in \cite{Euler} when he evaluated the zeta function $\zeta(s)$ at negative integers. It turns out that the coefficients of the $x^k$ term in $A_n(x)$, written $A_{n,k}$ and called \textit{Eulerian numbers}, can be interpretted combinatorially.

\begin{definition}
A permutation $\pi \in S_n$ has a \textit{descent} at position $i$ if $\pi(i) > \pi(i+1)$, where $i = 1, \dots, n-1$, and the \textit{descent set} of $\pi$, denoted $Des(\pi)$ is the set of all descents of $\pi$. The \textit{descent number} of $\pi$ is defined as $d(\pi) := \lvert Des(\pi) \rvert$.
\end{definition}

The results of MacMahon and Riordan, in \cite{MacMahon} and \cite{Riordan} respectively, showed that $A_{n,k}$ is the number of permutations in $S_n$ with $k$ descents.

The theory of descents in permutations has been studied thoroughly and is related to many questions. In \cite{Knuth}, Knuth connected descents with the theory of sorting and the theory of runs in permutations, and in \cite{Diaconis1}, Diaconis, McGrath, and Pitman studied a model of card shuffling in which descents play a central role. Bayer and Diaconis also used descents and rising sequences to give a simple expression for the chance of any arrangement after any number of shuffles and used this to give sharp bounds on the approach to randomness in \cite{Bayer}. Garsia and Gessel found a generating function for the joint distribution of descents, major index, and inversions in \cite{Garsia}, and Gessel and Reutenauer showed that the number of permutations with given cycle structure and descent set is equal to the scalar product of two special characters of the symmetric group in \cite{Gessel}. Diaconis and Graham also explained Peirce's dyslexic principle using descents in~\cite{DiaconisGraham}. Petersen also has an excellent and very thorough book on Eulerian numbers \cite{Petersen}.

It is well known (\cite{Diaconis2}) that the distribution of $d(\pi)$ in $S_n$ is asymptotically normal with mean $\frac{n+1}{2}$ and variance $\frac{n-1}{12}$. Fulman also used Stein's method to show that the number of descents of a random permutation satisfies a central limit theorem with error rate $n^{-1/2}$ in \cite{Fulman2}. In \cite{Vatutin}, Vatutin proved a central limit theorem for $d(\pi) + d(\pi^{-1})$, where $\pi$ is a random permutation.

Using generating functions, Fulman proved the following analogous result in \cite{Fulman1} about conjugacy classes with large cycles only:

\begin{theorem}
For every $n \geq 1$, pick a conjugacy class $C_n$ in $S_n$, and let $n_i(C_n)$ be the number of $i$-cycles in $C_n$. Suppose that for all $i$, $n_i(C_n) \to 0$ as $n \to \infty$. Then, the distribution of $d(\pi)$ in $C_n$ is asymptotically normal with mean $\frac{n-1}{2}$ and variance $\frac{n+1}{12}$.
\end{theorem}

Kim also used generating functions, in \cite{Kim}, to prove the following central limit theorem about the conjugacy class of fixed point free involutions:

\begin{theorem}
For every $n \geq 1$ even, let $C_n$ be the conjugacy class of fixed point free involutions in $S_n$. Then, the distribution of $d(\pi)$ in $C_n$ is asymptotically normal with mean $\frac{n}{2}$ and variance $\frac{n}{12}$.
\end{theorem}

After the above result was proved, Diaconis conjectured that there are asymptotic normality results for conjugacy classes that are fixed point free. In this paper, we will prove a generalized version of this conjecture that proves asymptotic normality of descents for all conjugacy classes of $S_n$.

\begin{theorem} \label{clt}
    Let $\pi$ be uniformly drawn form a conjugacy class $\mathcal{C}_{\lambda}$ having $m_1$ fixed points and $W_{\lambda}$ be the normalized descent number of $\pi$, in the sense that $d(\pi) = \frac{n+1}{2} - \frac{m_1^2}{2n} + \sqrt{n} W_\lambda$. Then, if $m_1/n \to \alpha$, $W_{\lambda}$ converges to $\mathcal{N}\left( 0, \frac{1 - 4\alpha^3 + 3\alpha^4}{12} \right)$ in distribution.
\end{theorem}

The outline is as follows. In Section~2, we expand the generating function $A_{\mathcal{C}_\lambda}(t)$ at infinity to obtain a series expression that is convergent for $\lvert t \rvert > 1$. In Section~3, we calculate the asymptotic variance of descent numbers of permutations, chosen uniformly at random, from a conjugacy class $\mathcal{C}_\lambda$, where $\lambda \vdash n$. Finally, in Section~4, we prove the following main theorem on the moment generating function $M_{\lambda}$ of the normalized descent numbers.
\begin{theorem} \label{main}
     Write $\alpha_{\lambda} = m_1/n$. Then, there exists a function $C : \mathbb{R} \to (0, \infty)$ such that
    \begin{equation*}
        \left| M_{\lambda}(s) - \exp\left\{ \frac{s^2}{24} \left( 1 - 4\alpha_{\lambda}^3 + 3\alpha_{\lambda}^4 \right) \right\} \right|
        \leq C(s) \frac{\log^3 n}{\sqrt{n}}
    \end{equation*}
    for any $n \geq 1$ and for any $\lambda \vdash n$.
\end{theorem}
\noindent We obtain the asymptotic normality as a consequence.

\section{Crossing the singularity}

Let $\mathcal{C}_\lambda$ be a conjugacy class of $S_n$, where $\lambda \vdash n$ and $\lambda$ consists of the cycle lengths. For each subset $S \subseteq S_n$, define the generating function $A_S(t) = \sum_{\pi \in S} t^{d(\pi)}$. In \cite{Fulman1}, Fulman showed that, if $\lambda$ has $m_i$ $i$'s,
\begin{equation}
	A_{\mathcal{C}_\lambda}(t) = (1-t)^{n+1} \sum_{a=1}^\infty t^a \prod_{i=1}^n \binom{f_{i,a} + m_i - 1}{m_i},\label{6.0}
\end{equation}
where $f_{i,a} = \frac{1}{i} \sum_{d \mid i} \mu(d) a^{i/d}$ and $\mu(d)$ is the M{\"o}bius function. This identity holds as a formal power series, and as an actual convergent series for $\lvert t \rvert < 1$. In~\cite{Reutenauer}, Reutenauer showed that $f_{i,a}$ counts the number of primitive circular words of length~$i$ from the alphabet $\{ 1, \dots, a \}$.

Recall that the moment generating function (MGF) of a random variable $X$ is defined by $M_X(s) = \mathbb{E}[ e^{sX} ]$. In~\cite{Kim}, Kim observed that we can construct $M_n(s)$, the MGF of descents in fixed point free involutions, from~(\ref{6.0}), by the relation $M_n(s) = A_{\mathcal{C}_\lambda}\left( e^s \right)$. After showing that the MGF converges pointwise to $e^{s^2/24}$, which is the MGF of a normal distribution, the desired central limit theorem followed from the pointwise convergence and the following result of Curtiss from~\cite{Curtiss}.

\begin{theorem}\label{curtiss}
Suppose we have a sequence $\left\{ X_n \right\}_{n=1}^\infty$ of random variables and there exists $s_0 > 0$ such that each MGF $M_n(s) = \mathbb{E}\left[ e^{s X_n} \right]$ converges for $s \in \left( -s_0, s_0 \right)$. If $M_n(s)$ converges pointwise to some function $M(s)$ for each $s \in \left( -s_0, s_0 \right)$, then $M$ is the MGF of some random variable $X$, and $X_n$ converges to $X$ in distribution.
\end{theorem}

However, since~(\ref{6.0}) is convergent for $\lvert t \rvert < 1$, the above relation for $M_n(s)$ is only convergent for $s < 0$. Fortunately, the descents of fixed point free involutions have a crucial palindromic property, also proven in~\cite{Kim}, which implies $M_n(s) = M_n(-s)$, and so, the pointwise convergence follows for all~$s$.

It turns out that descents of other conjugacy classes of $S_n$ do not have this palindromic property. Hence, we need an expression for $A_{\mathcal{C}_\lambda}(t)$, similar to~(\ref{6.0}), that converges for $\lvert t \rvert > 1$, in order to deal with MGF for $s < 0$. We claim the following proposition.

\begin{proposition}
Let $\mathcal{C}_\lambda$ be a conjugacy class of $S_n$. Then, for $\lvert t \rvert > 1$,
\begin{equation}
	A_{\mathcal{C}_\lambda}(t) = (t-1)^{n+1} \sum_{a=1}^\infty t^{-a} \left[ (-1)^n \prod_{i=1}^n \binom{f_{i,-a}+m_i-1}{m_i} \right].\label{6.1}
\end{equation}
\end{proposition}

In order to prove the proposition, we first prove an analogous statement for $A_n(t) = A_{S_n}(t)$.

\begin{lemma}
	For $\lvert t \rvert > 1$,
\begin{equation}
	A_n(t) = (t-1)^{n+1} \sum_{a=1}^\infty a^n t^{-a}.\label{6.2}
\end{equation}
\end{lemma}

\begin{proof}
Recall that $a^n = \sum_{k=0}^n \stirlingii{n}{k} (a)_k$, where $\stirlingii{n}{k}$ is the (unsigned) Stirling number of the second kind, and $(a)_k = a(a-1) \cdots (a-k+1)$ is the falling factorial. Plugging this identity into $A_n(t) = (1-t)^{n+1} \sum_{a=1}^\infty a^n t^a$, we see that
\begin{equation*}
	A_n(t) = (1-t)^{n+1} \sum_{k=0}^n \stirlingii{n}{k} \sum_{a=0}^\infty t^a (a)_k = \sum_{k=0}^n \stirlingii{n}{k} k! t^k (1-t)^{n-k},
\end{equation*}
where the second equality follows from
\begin{equation*}
	\sum_{a=0}^\infty (a)_k t^{a-k} = \left( \frac{d}{dt} \right)^k (1-t)^{-1} = k!(1-t)^{-(k+1)}.
\end{equation*}

Now, we can view $A_n(t)$ as a polynomial in $t$, and assuming $\lvert t \rvert > 1$, we substitute $t = 1/s$ to get
\begin{align*}
	A_n\left( \frac{1}{s} \right) &= (-1)^n s \left( \frac{1}{s} - 1 \right)^{n+1} \sum_{k=0}^n \stirlingii{n}{k} \frac{(-1)^k k!}{(1-s)^{k+1}} \\
	&= (-1)^n s \left( \frac{1}{s} - 1 \right)^{n+1} \sum_{k=0}^n \stirlingii{n}{k} (-1)^k \sum_{a=0}^\infty (a+k)_k s^a \\
	&= (-1)^n s \left( \frac{1}{s} - 1 \right)^{n+1} \sum_{a=0}^\infty \left( \sum_{k=0}^n \stirlingii{n}{k} (-1)^k (a+k)_k \right) s^a.
\end{align*}
We can simplify the expression further by noting $(-1)^k(a+k)_k = (-a-1)_k$, and so,
\begin{equation*}
	A_n \left( \frac{1}{s} \right) = (-1)^n s \left( \frac{1}{s} - 1 \right)^{n+1} \sum_{a=0}^\infty (-a-1)^n s^a = \left( \frac{1}{s} - 1 \right)^{n+1} \sum_{a=1}^\infty a^n s^a.
\end{equation*}
Plugging back $s = 1/t$ proves (\ref{6.2}).
\end{proof}

\begin{proof}[Proof of Proposition \ref{6.1}]
By viewing
\begin{equation*}
\prod_{i=1}^n \binom{f_{i,a}+m_i-1}{m_i} = \sum_{k=1}^n c_k a^k,
\end{equation*}
we can write
\begin{equation*}
A_{\mathcal{C}_\lambda} (t) = \sum_{k=1}^n c_k (1-t)^{n-k} A_k(t).
\end{equation*}
Hence, by (\ref{6.2}), for $\lvert t \rvert > 1$, we have
\begin{align*}
	A_{\mathcal{C}_\lambda}(t) &= \sum_{k=1}^n c_k (1-t)^{n-k} (t-1)^{k+1} \sum_{a=1}^\infty a^k t^{-a} \\
	&= (t-1)^{n+1} \sum_{a=1}^\infty t^{-a} \left[ (-1)^n \sum_{k=1}^\infty c_k(-a)^k \right] \\
	&= (t-1)^{n+1} \sum_{a=1}^\infty t^{-a} \left[ (-1)^n \prod_{i=1}^n \binom{f_{i,-a}+m_i-1}{m_i} \right]
\end{align*}
\end{proof}

\section{Computation of the asymptotic variance}

In~\cite{Fulman1}, Fulman showed that the asymptotic mean of the descent numbers of $\mathcal{C}_{\lambda}$ is
\begin{equation*}
	(1-\alpha^2) \frac{n}{2}
\end{equation*}
as $n \to \infty$ and $m_1/n \to \alpha$, by analyzing~(\ref{6.0}). Using similar methods, we calculate the asymptotic variance.

\begin{lemma}
The asymptotic variance of the descent numbers of $\mathcal{C}_{\lambda}$ is
\begin{equation*}
	\left( 1 - 4\alpha^3 + 3\alpha^4 \right) \frac{n}{12}
\end{equation*}
as $n \to \infty$ and $m_1/n \to \alpha$.
\end{lemma}

\begin{proof}
From~(\ref{6.0}), we see that
\begin{align*}
	\frac{A_{\mathcal{C}_\lambda}(t)}{\lvert \mathcal{C}_\lambda \rvert}
	&= \frac{(1-t)^{n+1}}{n!} \sum_{a=0}^\infty t^a \prod_{i=1}^n \Bigg( \sum_{k=1}^{m_i} \sum_{d_1,\dots,d_k \mid i} i^{m_i - k} \stirlingi{m_i}{k} \\
	&\hspace{15em} \times \mu\left( \frac{i}{d_1} \right) \cdots \mu\left( \frac{i}{d_k} \right) a^{d_1 + \cdots + d_k} \Bigg),
\end{align*}
where $\stirlingi{n}{k}$ denotes the Stirling numbers of the first kind. The asymptotic mean in~\cite{Fulman1} was calculated by noting that $d_1 + \cdots + d_k = i m_i$ if and only if $k = m_1$ and $d_1 = \cdots d_k = i$, and also noting that $d_1 + \cdots + d_k = i m_i - 1$ if and only if one of the following is true:
\begin{enumerate}
	\item $i=2$, $k=m_2$, and $\left\{ d_1, \dots, d_k \right\} = \left\{ 1, 2, \dots, 2 \right\}$ as multisets.
	\item $i=1$, $k=m_1-1$, and $\left\{ d_1, \dots, d_k \right\} = \left\{ 1, \dots, 1 \right\}$ as multisets.
\end{enumerate}
Similarly, we note that $d_1 + \cdots + d_k = i m_i - 2$ if and only if one of the following is true:
\begin{enumerate}
	\item $i=4$, $k=m_4$, and $\left\{ d_1, \dots, d_k \right\} = \left\{ 2, 4, \dots, 4 \right\}$ as multisets.
	\item $i=3$, $k=m_3$, and $\left\{ d_1, \dots, d_k \right\} = \left\{ 1, 3, \dots, 3 \right\}$ as multisets.
	\item $i=2$, $k=m_2$, and $\left\{ d_1, \dots, d_k \right\} = \left\{ 1, 1, 2, \dots, 2 \right\}$ as multisets.
	\item $i=2$, $k=m_2-1$, and $\left\{ d_1, \dots, d_k \right\} = \left\{ 2, \dots, 2 \right\}$ as multisets.
	\item $i=1$, $k=m_1-2$, and $\left\{ d_1, \dots, d_k \right\} = \left\{ 1, \dots, 1 \right\}$ as multisets.
\end{enumerate}
Hence, it follows that
\begin{align*}
	\frac{A_{\mathcal{C}_\lambda}(t)}{\lvert \mathcal{C}_\lambda \rvert}
	&= \frac{A_n(t)}{n!} + \frac{1-t}{n} \frac{A_{n-1}(t)}{(n-1)!} \left( \binom{m_1}{2} - m_2 \right) \\
	&\quad + \frac{(1-t)^2}{n(n-1)} \frac{A_{n-2}(t)}{(n-2)!} \left( \frac{3m_1 - 1}{4} \binom{m_1}{3} + 3 \binom{m_2}{2} - m_2 \binom{m_1}{2} - m_3 - m_4 \right) \\
	&\quad + (1-t)^3 g(t)
\end{align*}
for some polynomial $g(t)$, from which we can calculate the asymptotic variance of descent numbers to be $\left( 1 - 4\alpha^3 + 3\alpha^4 \right)\frac{n}{12}$.
\end{proof}

\section{Central Limit Theorem for Descents in Conjugacy Classes of $S_n$}

Write $D_{\lambda}$ for the descent number $d(\pi)$ of a permutation $\pi$ which is uniformly chosen from the conjugacy class $\mathcal{C}_\lambda$ of $S_n$. Let us define the normalized random variable~$W_\lambda$ by
\begin{equation*}
    D_\lambda = \frac{n+1}{2} - \frac{m_1^2}{2n} + \sqrt{n} W_\lambda,
\end{equation*}
and denote by $M_\lambda(s) = \mathbb{E}[e^{s W_{\lambda}}]$ the MGF of $W_\lambda$.

Since we now know the asymptotic mean and variance, we expect that the distribution of $W_{\lambda}$ is asymptotically the normal distribution with the zero mean and the variance $\frac{1}{12}(1 - 4\alpha^3 + 3\alpha^4)$ along the limit $m_1 / n \to \alpha$. This is equivalent to showing that $M_{\lambda}$ converges pointwise to the MGF of this normal distribution as $m_1/n \to \alpha$. The following result, Theorem~\ref{main}, concerns a uniform estimate on $M_{\lambda}$ that serves this purpose.

{
\renewcommand{\thetheorem}{1.5}
\begin{theorem}
    Write $\alpha_{\lambda} = m_1/n$. Then, there exists a function $C : \mathbb{R} \to (0, \infty)$ such that
    \begin{equation*}
        \left| M_{\lambda}(s) - \exp\left\{ \frac{s^2}{24} \left( 1 - 4\alpha_{\lambda}^3 + 3\alpha_{\lambda}^4 \right) \right\} \right|
        \leq C(s) \frac{\log^3 n}{\sqrt{n}}
    \end{equation*}
    for any $n \geq 1$ and for any $\lambda \vdash n$.
\end{theorem}
\setcounter{theorem}{0}
}

This section is aimed at proving this theorem. In Section~4.1, we develop a series representation of $M_{\lambda}$ along with some preliminary estimates on its coefficients. This reduces the main claim to Proposition \ref{mainprop}. In Section~4.2, we prove this proposition.

\subsection{Series representation of $M_{\lambda}$}

In~\cite{Kim}, estimating the coefficients of \eqref{6.0} played an important role in the computation. Likewise, we need to provide an estimation on the coefficients of \eqref{6.1}. Let $F_{i,a} = i f_{i,a} = \sum_{i \mid d} \mu(i) a^{d/i}$. We first prove the following lemma about $F_{i,a}$ and $F_{i,-a}$.

\begin{lemma}\label{F-lemma}
	Let $a$ and $i$ be positive integers. Then,
	\begin{enumerate}
		\item $(-1)^i F_{i,-a} = F_{i,a} + 2F_{\frac{i}{2},a}\mathbbm{1}_{\left\{\text{ord}_2(i)=1\right\}}$,
		\item (upper bound) $0 \leq F_{i,a} \leq a^i$ and $0 \leq (-1)^i F_{i,-a} \leq a^i + 2a^{\frac{i}{2}}$, and
		\item (lower bound) $(-1)^i F_{i,-a} \geq F_{i,a} \geq a^{\frac{i}{2}} \left( a^{\frac{i}{2}} - \frac{i}{2} \right)$.
	\end{enumerate}
\end{lemma}

\begin{proof}[Proof of Lemma~\ref{F-lemma}]
    Part (1) is proven by looking at the definition of $f_{i,a}$. Let us write $i = 2^k q$, where $k$ is a positive integer and $q$ is an odd integer. Then, by the multiplicity of $\mu$, we have
    \begin{equation*}
    	F_{i,a} = \sum_{j=0}^k \sum_{d \mid q} \mu(d) \mu\left( 2^j \right) a^{2^{k-j}\frac{q}{d}},
    \end{equation*}
    and
    \begin{equation*}
    	F_{i,-a} = \sum_{j=0}^k \sum_{d \mid q} \mu(d) \mu\left( 2^j \right) (-a)^{2^{k-j}\frac{q}{d}}.
    \end{equation*}
    
    We divide the computation into three cases.
    \begin{enumerate}
    	\item If $k = 0$, $i = q$ is odd, and so,
    	\begin{equation*}
    		(-1)^i F_{i,-a} = -\sum_{d \mid q} \mu(d) (-a)^{\frac{q}{d}} = \sum_{d \mid q} \mu(d) a^{\frac{q}{d}} = F_{i,a}.
    	\end{equation*}
    	\item If $k = 1$, both the terms for $j = 0,1$ may survive, and
    	\begin{align*}
    		(-1)^i F_{i,-a} &= \sum_{d \mid q} \mu(d) (-a)^{\frac{2q}{d}} + \sum_{d \mid q} \mu(2) \mu(d) (-a)^{\frac{q}{d}} \\
    		&= \sum_{d \mid q} \mu(d) a^{\frac{2q}{d}} + \sum_{d \mid q} \mu(d) a^{\frac{q}{d}} \\
    		&= \sum_{d \mid 2q} \mu(d) a^{\frac{2q}{d}} + 2\sum_{d \mid q} \mu(d) a^{\frac{q}{d}} \\
    		&= F_{i,a} + 2F_{\frac{i}{2}, a}.
    	\end{align*}
    	\item If $k \geq 2$, we have $(-1)^i = 1$ and $(-a)^{2^{k-j}\frac{q}{d}} = a^{2^{k-j}\frac{q}{d}}$ for $j = 0,1$ and $d \mid q$. Hence, by comparing the formula for $F_{i,a}$ and $(-1)^i F_{i,-a}$, we see that they coincide.
    \end{enumerate}
    
    For part (2), we note that $f_{i,a}$ counts certain types of words, and so, $F_{i,a} = if_{i,a} \geq 0$. By using M{\"o}bius inversion formula, we see that $F_{i,a} \leq \sum_{d \mid i} F_{i,a} = a^i$. The second inequality follows from part (1) and the first inequality.
    
    For part (3), note that, by parts (1) and (2), we have $(-1)^i F_{i,-a} \geq F_{i,a}$. The other half of the inequality follows by noting that
    \begin{equation*}
    	F_{i,a} \geq a^i - \sum_{d \mid i, d \neq i} a^d \geq a^i - \frac{i}{2} a^{\frac{i}{2}},
    \end{equation*}
    and so, the lemma is proven.
\end{proof}

In order to utilize both representations \eqref{6.0} and \eqref{6.1} simultaneously, we introduce some auxiliary notations as follows. Given a partition $\lambda \vdash n$ and a non-zero real number $s$, define
\begin{align*}
    K_{a}
    = \prod_{i=1}^{n} K_{a}^{(i)},
    \qquad
    K_{a}^{(i)}
    = \begin{cases}
        \displaystyle \prod_{k=0}^{m_{i}-1} (F_{i,a} + ik), & \text{if } s < 0 \\
        \displaystyle \prod_{k=0}^{m_{i}-1} (-1)^i (F_{i,-a} + ik), & \text{if } s > 0
    \end{cases}
\end{align*}
for $1 \leq i \leq n$ and $a \geq 1$. Strictly speaking, both $K_{a}$ and $K_{a}^{(i)}$ depend on both $s$ and $\lambda$ as well. Since $s$ and $\lambda$ are assumed to be given throughout the computation, however, we suppress them from the notation. Then by \eqref{6.0} and \eqref{6.1}, we obtain the following concise formula
\begin{align*}
	\mathbb{E}[e^{sD_\lambda}]
	= \frac{A_{\mathcal{C}_{\lambda}}(e^{s})}{|\mathcal{C}_{\lambda}|}
	= \left( \frac{e^s-1}{s} \right)^{n+1} \frac{|s|^{n+1}}{n!} \sum_{a=1}^\infty K_a e^{-|s|a}.
\end{align*}
From this, we find that $M_{\lambda}$ is given by
\begin{align*}
    M_{\lambda}(s)
	&= \mathbb{E} \exp\left\{ \frac{s}{\sqrt{n} } \left( D_{\lambda} - \frac{n+1}{2} + \frac{m_1^2}{2n} \right) \right\} \\
	&= \left( \frac{\sinh\left( \frac{s}{2\sqrt{n}} \right)}{\frac{s}{2\sqrt{n}}} \right)^{n+1} \frac{\left( |s|/\sqrt{n} \right)^{n+1}}{n!} \sum_{a=1}^{\infty} K_a \exp \left\{ -\frac{|s|}{\sqrt{n}}a + \frac{m_1^2 s}{2n^{3/2}} \right\} .
\end{align*}
For the sake of simplicity, let us denote
\begin{align*}
    L_a
    = \frac{\left( |s|/\sqrt{n} \right)^{n+1}}{n!} K_a \exp \left\{ -\frac{|s|}{\sqrt{n}}a + \frac{m_1^2 s}{2n^{3/2}} \right\}.
\end{align*}
Note that, for $s$ fixed and $n \to \infty$,
\begin{align*}
    \left( \frac{\sinh\left( \frac{s}{2\sqrt{n}} \right)}{\frac{s}{2\sqrt{n}}} \right)^{n+1}
    = \left( 1 + \frac{s^2}{24n} + \mathcal{O} \left( \frac{1}{n^2} \right) \right)^{n+1}
    = e^{\frac{s^2}{24}} + \mathcal{O}\left( \frac{1}{n} \right),
\end{align*}
where the implicit bounds depend only on $s$. In light of this, we have only to prove the following proposition.

\begin{proposition} \label{mainprop}
Write $\alpha_\lambda = m_1/n$. Then, there exists a function $C: \mathbb{R} \to (0,\infty)$ such that
\[
    \left\lvert \sum_{a=1}^\infty L_a - \exp\left\{ \frac{s^2}{24} \left( -4\alpha_\lambda^3 + 3\alpha_\lambda^4 \right) \right\} \right\rvert \leq C(s) \frac{\log^3 n}{\sqrt{n}}.
\]
\end{proposition}

\subsection{Proof of Proposition \ref{mainprop}}

We inspect the sum over two ranges -- the small range, where $a \leq \varepsilon n^{3/2}$, and the large range, where $a > \varepsilon n^{3/2}$. Here, $\varepsilon$ is a positive real number chosen to satisfy $4\varepsilon e |s| < 1$. This choice will be explained shortly later, but it is important to note that $\varepsilon$ depends only on $s$.

When invoking asymptotic notation $\mathcal{O}(\cdot)$, it is always assumed that implicit bounds depend only possibly on $s$. This way, we can keep track of uniform estimates. Likewise, we indulge in luxury of changing the meaning of the generic function $C = C(s)$ from line to line, as its exact values are not important to the argument.

\subsubsection{Estimation of the small range.}

If $a \leq \varepsilon n^{3/2}$, we have
\begin{align*}
	\lvert K_a \rvert
	\leq \prod_{i=1}^n \prod_{k=0}^{m_i-1} \left( a^i + 2a^{\frac{i}{2}} + ik \right)
	\leq \prod_{i=1}^n \prod_{k=0}^{m_i-1} \left( 3 \varepsilon n^{\frac{3}{2}} + n \right)
    \leq \max\left\{ 4n, 4\varepsilon n^{\frac{3}{2}} \right\}^n,
\end{align*}
where the last inequality follows from bounding each factor $\left( 3 \varepsilon n^{\frac{3}{2}} + n \right)$ by $4$~times the bigger of $n$ and $\varepsilon n^{3/2}$. This induces the following upper bound of $L_{a}$.
\begin{align*}
    \lvert L_a \rvert
    \leq \frac{\left( |s|/\sqrt{n} \right)^{n+1}}{n!} \max\left\{ 4n, 4\varepsilon n^{\frac{3}{2}} \right\}^n e^{\sqrt{n}|s|/2}.
\end{align*}
Taking a union bound, it follows that
\begin{align*}
    \sum_{a \leq \varepsilon n^{3/2}} L_a
    \leq \varepsilon n^{\frac{3}{2}} \left( \max_{ a \leq \varepsilon n^{3/2} } |L_a| \right)
    \leq \frac{\varepsilon |s|n}{n!} \max\left\{ 4|s|\sqrt{n}, 4\varepsilon|s| n \right\}^n e^{\sqrt{n}|s|/2}.
\end{align*}
In view of the Stirling's approximation $n! \sim \sqrt{2\pi} n^{n+\frac{1}{2}} e^{-n}$, this bound decays to $0$ at least as exponentially fast as $n \to \infty$ by our choice of $\varepsilon$.

\subsubsection{Estimation of the large range.}

Through this section, we assume that $a > \varepsilon n^{\frac{3}{2}}$. In this range, we first check that $K_{a}^{(i)}$, for $2 \leq i \leq n$, behaves almost the same as $a^{im_i}$. More precisely, fix $N_1 = N_1(s)$ so that $\frac{1}{2\varepsilon n^{1/2}} + \frac{1}{\varepsilon^2 n^2} < \frac{1}{2}$ for all $n \geq N_1$, which we assume hereafter. Then, by Lemma \ref{F-lemma}, we find that
\begin{align*}
    0
    < \prod_{k=0}^{m_i-1} \left( 1 - \frac{i}{2a^{i/2}} - \frac{ik}{a^i} \right)
    \leq \frac{K_a^{(i)}}{a^{i m_i}}
    \leq \prod_{k=0}^{m_i-1} \left( 1 + \frac{2}{a^{i/2}} + \frac{ik}{a^i} \right)
\end{align*}
for all $2 \leq i \leq n$. Applying the estimate $\log(1+x) = \mathcal{O}(x)$ for $|x| \leq \frac{1}{2}$, we have
\begin{align*}
	\log \left( \frac{1}{a^{n-m_1}} \prod_{i=2}^n K_a^{(i)} \right)
	&\leq C \sum_{i=2}^{n} \sum_{k=0}^{m_i-1} \left( \frac{i}{a^{i/2}} + \frac{ik}{a^i} \right)
	 \leq C \sum_{i=2}^{n} \sum_{k=0}^{m_i-1} \left( \frac{i}{a} + \frac{ik}{a^2} \right) \\
	&\leq C \left( \frac{n}{a} + \frac{n^2}{a^2} \right)
	 \leq \frac{C}{\sqrt{n}},
\end{align*}
As for $K_a^{(1)}$, we need to consider both $s>0$ and $s<0$ cases. If $s<0$, then by the expansion $1+x = \exp\{\log(1+x)\} = \exp\left\{ x - \frac{1}{2}x^2 + \mathcal{O}(x^3) \right\}$ we get
\begin{align*}
	a^{-m_1} K_a^{(1)}
	&= \prod_{k=0}^{m_1-1} \left( 1 + \frac{k}{a} \right)
    = \exp \left\{ \sum_{k=0}^{m_1-1} \left( \frac{k}{a} - \frac{k^2}{2a^2} + \mathcal{O}\left( \frac{k^3}{a^3} \right) \right) \right\} \\
	&= \exp \left\{ \frac{m_1^2}{2a} - \frac{m_1^3}{6a^2} + \mathcal{O}\left( \frac{1}{\sqrt{n}} \right) \right\}.
\end{align*}
Similarly, for $s>0$, we get
\begin{align*}
	a^{-m_1} K_a^{(1)}
	= \prod_{k=0}^{m_1-1} \left( 1 - \frac{k}{a} \right)
	= \exp \left\{ -\frac{m_1^2}{2a} - \frac{m_1^3}{6a^2} + \mathcal{O}\left( \frac{1}{\sqrt{n}} \right) \right\}.
\end{align*}
Combining the results, we see that, for $s \neq 0$,
\begin{align*}
	K_a
	= a^n \exp \left\{ -\frac{m_1^2}{2a} \sgn(s) - \frac{m_1^3}{6a^2} + \mathcal{O}\left( \frac{1}{\sqrt{n}} \right) \right\},
\end{align*}
where the implicit constant in $\mathcal{O}\left( \frac{1}{n} \right)$ depends only on $\varepsilon$. From this, it easily follows that, for $x \in [a, a+1]$,
\begin{align*}
    K_a \exp\left\{ -\frac{|s|}{\sqrt{n}} a \right\}
	&= \left( 1 + \mathcal{O}\left( \frac{1}{\sqrt{n}} \right) \right) x^n \exp \left\{ -\frac{m_1^2}{2x} \sgn(s) - \frac{m_1^3}{6x^2} - \frac{|s|}{\sqrt{n}} x \right\}
\end{align*}
and hence,
\begin{align*}
    \sum_{a > \varepsilon n^{3/2}} L_a
    &= \left( 1 + \mathcal{O} \left( \frac{1}{\sqrt{n}} \right) \right)
    \exp\left\{ \frac{m_1^2 s}{2n^{3/2}} \right\} \frac{\left( |s|/\sqrt{n} \right)^{n+1}}{n!} \\
    &\hspace{1.5em} \times  \int_{\varepsilon n^{3/2}}^{\infty} x^n \exp \left\{ -\frac{m_1^2}{2x} \sgn(s) - \frac{m_1^3}{6x^2} - \frac{|s|}{\sqrt{n}} x \right\} \, dx.
\end{align*}
Applying the substitution $y = \frac{|s|}{\sqrt{n}} x$, followed by $y = n + \sqrt{n} z$, the above integral becomes
\begin{align*}
    &\exp\left\{ \frac{m_1^2 s}{2n^{3/2}} \right\} \frac{\left( |s|/\sqrt{n} \right)^{n+1}}{n!} \int_{\varepsilon n^{3/2}}^{\infty} x^n \exp \left\{ -\frac{m_1^2}{2x} \sgn(s) - \frac{m_1^3}{6x^2} - \frac{|s|}{\sqrt{n}} x \right\} \, dx \\
    &\hspace{1em} = \exp\left\{ \frac{m_1^2 s}{2n^{3/2}} \right\} \frac{1}{n!} \int_{\varepsilon |s| n}^{\infty} y^n \exp \left\{ -\frac{m_1^2 s}{2\sqrt{n} y} - \frac{m_1^3 s^2}{6ny^2} - y \right\} \, dy \\
    &\hspace{1em} = \frac{n^{n+\frac{1}{2}} e^{-n}}{n!} \int_{\mathbb{R}} g_n(z) \exp \left\{ \frac{\alpha_{\lambda}^2 s }{2 ( 1 + \frac{z}{\sqrt{n}} )} z - \frac{\alpha_{\lambda}^3 s^2}{6 ( 1 + \frac{z}{\sqrt{n}} )^2} \right\} \, dz,
\end{align*}
where $\alpha_{\lambda} = m_1 / n$ and $g_n : \mathbb{R} \to \mathbb{R}$ is defined by
\begin{align*}
    g_n(z) = \begin{cases}
        \displaystyle \left( 1 + \frac{z}{\sqrt{n}} \right)^n e^{-\sqrt{n} z}, & \text{if } z > -(1-\varepsilon|s|)\sqrt{n} \\
        0, & \text{otherwise}.
    \end{cases}
\end{align*}

Now we aim at estimating the last integral. First, by the Stirling's approximation we have $\frac{n^{n+\frac{1}{2}} e^{-n}}{n!} = \frac{1}{\sqrt{2\pi}} + \mathcal{O}\left(\frac{1}{n} \right)$. Next, we claim the following lemma.

\begin{lemma} \label{g-lemma}
    Let $N \geq 1$ be arbitrary. Then, for any $n \geq N$, we have
    \begin{align*}
    	g_n(z)
	    \leq \begin{cases}
	        g_N(z), & \text{if $z \geq 0$} \vspace{0.5em} \\
	       e^{-\frac{1}{2}z^2}, & \text{if $z < 0$}
        \end{cases}.
    \end{align*}
\end{lemma}

\begin{proof}[Proof of Lemma~\ref{g-lemma}]
    Consider the function $h(n, z) = \log \left( \left( 1 + \frac{z}{\sqrt{n}} \right)^n e^{-\sqrt{n} z} \right)$ on $z > -\sqrt{n}$. By direct computation, we find that
    \begin{align*}
        \frac{\partial^2 h}{\partial n^2} = \frac{z^3}{4n^{3/2} \left( z + \sqrt{n} \right)^2}
        \qquad \text{and} \qquad
        \lim_{n\to\infty} \frac{\partial h}{\partial n} = 0.
    \end{align*}
   So, it follows that $\frac{\partial h}{\partial n} \leq 0$ if $z \geq 0$, and $\frac{\partial h}{\partial n} \geq 0$ if $z < 0$. Hence, for $n \geq N$ and $z \geq 0$, we obtain $g_n(z) = \exp\{h(n, z)\} \leq \exp\{h(N, z)\} = g_N(z)$. Similarly, when $z < 0$ we have $g_n(z) \leq \lim_{n'\to\infty} \exp\{h(n', z)\} = e^{-\frac{1}{2}z^2}$.
\end{proof}

Now, pick $N_2 = N_2(s)$ so that $N_2 \geq \max\{N_1, s^2\}$ (recall that we introduced $N_1$ at the beginning of the estimation in the large range.) Writing $\tilde{g}_n$ for the integrand
\begin{align*}
    \tilde{g}_n(z) = g_n(z) \exp \left\{ \frac{\alpha_{\lambda}^2 s }{2 ( 1 + \frac{z}{\sqrt{n}} )} z - \frac{\alpha_{\lambda}^3 s^2}{6 ( 1 + \frac{z}{\sqrt{n}} )^2} \right\},
\end{align*}
the above Lemma \ref{g-lemma} provides the following bound
\begin{align*}
    g_n(z) \leq
	\begin{cases}
	    g_{N_2}(z) e^{\frac{|s|}{2} z} & \text{if $z \geq 0$} \vspace{0.5em} \\
	    e^{-\frac{1}{2}z^2+\frac{|s|}{2}z} & \text{if $z < 0$}
    \end{cases}
\end{align*}
for all $z \in \mathbb{R}$ and for all $n \geq N_2$. The specific detail of this bound is not important, however, and we only need to note that this decays exponentially fast. To be precise, there exist constants $C > 0$ and $c > 0$, which depend only on $s$, such that
\begin{align*}
     \max \left\{ g_{N_2}(|z|) e^{\frac{|s|}{2} |z|}, e^{-\frac{1}{2}z^2+\frac{|s|}{2}|z|} \right\}
     \leq C e^{-c |z|}
\end{align*}
for all $z \in \mathbb{R}$. Now, to estimate the integral of $\tilde{g}_n$, we split this into two parts
\begin{align*}
    \int_{\mathbb{R}} \tilde{g}_n(z)
    = \int_{|z| \leq \frac{\log n}{2c}} \tilde{g_n}(z) \, dz
    + \int_{|z| > \frac{\log n}{2c}} \tilde{g_n}(z) \, dz.
\end{align*}
The latter integral is easily estimated by direct computation.
\begin{align*}
    \int_{|z| > \frac{\log n}{2c}} \tilde{g_n}(z) \, dz
    \leq 2 \int_{\frac{\log n}{2c}}^{\infty} Ce^{-cz} \, dz
    = \frac{2C}{c\sqrt{n}}.
\end{align*}
For the first integral, we have
\begin{align*}
    \int_{|z| \leq \frac{\log n}{2c}} \tilde{g_n}(z) \, dz
    &=  \int_{|z| \leq \frac{\log n}{2c}} \exp\left\{ - \frac{1}{2}z^2 + \frac{\alpha_{\lambda}^2 s}{2}z - \frac{\alpha_{\lambda}^3 s^2}{6} + \mathcal{O}\left( \frac{\log^3 n}{\sqrt{n}} \right) \right\} \, dz \\
    &= \left( 1 + \mathcal{O} \left( \frac{\log^3 n}{\sqrt{n}} \right) \right) \int_{\mathbb{R}} \exp\left\{ - \frac{1}{2}z^2 + \frac{\alpha_{\lambda}^2 s}{2}z - \frac{\alpha_{\lambda}^3 s^2}{6} \right\} \, dz \\
    &\hspace{2em} + \mathcal{O} \left( \int_{|z| > \frac{\log n}{2c}} C e^{-c|z|} \, dz \right) \\
    &= \sqrt{2\pi} \exp \left\{\frac{s^2}{24} \left( -4\alpha_{\lambda}^3 + 3\alpha_{\lambda}^4 \right) \right\} + \mathcal{O} \left( \frac{\log^3 n}{\sqrt{n}} \right).
\end{align*}
Combining altogether, we obtain, for $n \geq N_2$,
\begin{align*}
    \sum_{a > \varepsilon n^{3/2}} L_a
    = \exp \left\{\frac{s^2}{24} \left( -4\alpha_{\lambda}^3 + 3\alpha_{\lambda}^4 \right) \right\} + \mathcal{O} \left( \frac{\log^3 n}{\sqrt{n}} \right).
\end{align*}

Together with the exponential decay of $\sum_{a \leq \varepsilon n^{3/2}} L_a$ proved in the previous section, the desired proposition follows by revealing the implicit bound $C$ and then making it larger, if needed, so that the inequality is also true for $n < N_2$.

This concludes the proof of Proposition~\ref{mainprop} and, in turn, Theorem~\ref{main}. Combining Theorem~\ref{main} and Theorem~\ref{curtiss} yields the desired central limit theorem, Theorem~\ref{clt}.

\section*{Acknowledgments}

The authors would like to thank Jason Fulman and Persi Diaconis for suggesting the original problem.
Sangchul Lee's research has been partially supported by the NSF award DMS-1712632.

\bibliographystyle{amsplain}

\end{document}